\documentclass[a4paper,12pt]{article}

\usepackage{comment,amsmath,amsfonts,dsfont,amssymb,amsthm,mathrsfs,stmaryrd,enumitem}
\usepackage{tikz}
\usepackage[utf8]{inputenc}
\usepackage{centernot}
\usepackage{hyperref}
\usepackage[toc,page]{appendix}

\newtheorem{thm}{Theorem}[section]
\newtheorem{defi}[thm]{Definition}
\newtheorem{prop}[thm]{Proposition}

\newtheorem{lem}[thm]{Lemma}
\numberwithin{equation}{section}

\newcommand\nconnect{
  \mathrel{\ooalign{$\longleftrightarrow$\cr
    \hskip 0pt plus 0.8fill $\arrownot$ \hskip 0pt plus 2fill\cr}}
}
\newcommand\oset{\big\{\,}
\newcommand\cset{\,\big\}}
\newcommand{\fk}[3]{\Phi^{\, #1}_{#2} \big[ #3 \big] } 
 
\newcommand{\mfk}[3]{\Phi^{\, #1}_{#2} \left[ #3 \right] }

\newcommand{\cfk}[4]{\Phi^{\, #1}_{#2} \big[ #3 \bigm\arrowvert #4\, \big]}

\everymath{\displaystyle}
\begin{document}
\title{The localisation of low-temperature interfaces in $d$ dimensional Ising model}
\author{Wei Zhou\footnote{
\noindent
D\'epartement de math\'ematiques et applications, Ecole Normale Sup\'erieure,
CNRS, PSL Research University, 75005 Paris.
\newline
Laboratoire de Math\'ematiques d'Orsay, Universit\'e Paris-Sud, CNRS, Universit\'e
Paris--Saclay, 91405 Orsay.}
}
\maketitle
\begin{abstract}
We study the Ising model in a box $\Lambda$ in $\mathbb{Z}^d$ (not necessarily parallel to the directions of the lattice) with Dobrushin boundary conditions at low temperature. We couple the spin configuration with the configurations under $+$ and $-$ boundary conditions and we define the interface as the edges whose endpoints have the same spins in the $+$ and $-$ configurations but different spins with the Dobrushin boundary conditions. We prove that, inside the box $\Lambda$, the interface is localized within a distance of order $\ln^2|\Lambda|$ of the set of the edges which are connected to the top by a $+$ path and connected to the bottom by a $-$ path.
\end{abstract}
\section{Introduction}
At the macroscopic level, the dynamics of the interface between two pure phases in the Ising model seem to be deterministic. In fact, the interface tends to minimize the surface tension between the two phases. The microscopic justification of this fact in the context of 2D Ising model was achieved in \cite{MR1181197}. In the limit where the size of the system grows to infinity, the two types of spins, at low temperature, form two regions separated by the interfaces. After a suitable spatial rescaling, these interfaces converge to deterministic shapes. However, the interfaces remain random and their geometric structure is extremely complex. In two dimensions, the fluctuations of the interfaces are well analysed in \cite{DH} using the cluster expansions techniques. Recently, Ioffe and Velenik gave a geometric description of the interfaces and their scaling limits with the help of the Ornstein–Zernike theory in \cite{IV18}. In higher dimensions, the famous result of Dobrushin in \cite{D72Gibbs} says that at low temperature, the interface in a straight box is localised around the middle hyperplane of the box when the temperature is low. One of the difficulties to study the interfaces is to define them properly. The usual way is to consider the Dobrushin boundary conditions. More precisely, the vertices on the upper boundary of the box are pluses and those on the lower boundary are minuses. It is a geometric fact that, with such a boundary condition, the spin configurations present an interface separating a region of plus spins containing the upper boundary and a region of minus spins containing the lower boundary. However, for several reasons, it is still not obvious to define an interface in this setting. For example, there are more than one separating set between the pluses and minuses in a typical configuration with Dobrushin boundary conditions.
\vspace{-0.5cm}
\paragraph{Remark.}While I was finalizing this paper, Gheissari and Lubetzky completed a very interesting paper \cite{1901.04980} on the large deviations of the interface in 3D Ising model. They study the height of the interface in a straight box using a decomposition of the pillars and they obtain a localisation result at an order $\ln|\Lambda|$ at low temperature. Our localisation result is clearly weaker in the case of a straight box, however it holds also for a tilted box.

The first goal of this study is to adapt to the Ising model the definition of the interfaces, introduced in \cite{new} for the percolation model. 
The second goal is to progress in the geometric description of these interfaces for a box not necessarily straight in dimensions $d\geqslant 2$ at low temperature. In \cite{new}, we constructed a coupling between the dynamical percolation process and a conditioned process. The interface was defined as the difference between the two processes. We showed a localisation result for the interface around the pivotal edges for the disconnection event. For the Ising model, a coupling can be realised by the Glauber dynamics, yet there is no corresponding notion for the pivotal edges. However, the objects introduced for the percolation model in \cite{new} are defined naturally for the FK-percolation model. With the help of the Edwards-Sokal coupling, we can define and localise the interfaces using the results obtained in the FK-percolation model. To realise our first goal, we construct three spin configurations $(\sigma^+,\sigma^-,\sigma^D)$, corresponding to the plus, minus and Dobrushin boundary conditions, and a probability measure $\pi_{\Lambda,\beta}$ on this triplet, whose marginals are the Ising measures with the corresponding boundary conditions. We consider a box $\Lambda = (V,E)$ and we define the interface as follows:
\begin{defi}
The set $\mathcal{P}_I$ is the set of the edges $\langle x,y\rangle\in E$ such that 
$$\begin{array}{l}
\sigma^D(x) = +1\text{ and }x\text{ is connected to }T\text{ by a path of vertices with }+1\\
\sigma^D(y) = -1\text{ and }y\text{ is connected to }B\text{ by a path of vertices with }-1.
\end{array}
$$
The set $\mathcal{I}_I$ is the set of the edges $\langle x,y\rangle\in E$ such that
$$\sigma^{+}(x) = \sigma^{+}(y),\quad \sigma^-(x)=\sigma^-(y), \quad \sigma^D(x)\neq\sigma^D(y). 
$$
\end{defi}
\noindent The interface $\mathcal{I}_I$ is the set of the edges whose endpoints have different spins in $\sigma^D$ but have the same spins in the other two configurations. The set $\mathcal{P}_I$ corresponds to the edges of $\mathcal{I}_I$ connected to the boundary in $\sigma^D$. As for the second goal, we show the following result:
\begin{thm}\label{fk.main2}
There exist $0<\tilde{\beta}<\infty$ and $\kappa\geqslant 0$, such that for $\beta\geqslant \tilde{\beta}$, $c>0$ and any $\Lambda$ such that $|\Lambda|\geqslant \max\{3^{6d},(cd)^{cd^2}\}$, we have 
$$\pi_{\Lambda,\beta}\Big(\exists e \in \mathcal{I}_I,d(e,\Lambda^c\cup \mathcal{P}_I)\geqslant \kappa c^2 \ln^2|\Lambda|\Big)\leqslant \frac{1}{|\Lambda|^c}.
$$
\end{thm}
We call a cut in a spin configuration a set of edges $e = \langle x,y\rangle$ separating $T$ and $B$ such that 
$$ \sigma^D(x) \neq \sigma^D(y).
$$
Using the same method, we show that, under the probability $\pi$, a vertex separated from $B$ by a cut and which is far from this cut has the same spin in $\sigma^+$ and $\sigma^D$, more precisely, we have:
\begin{thm}\label{fk.main3}
There exist $0<\tilde{\beta}<\infty$ and $\kappa\geqslant 0$, such that for $\beta\geqslant \tilde{\beta}$, $c>0$ and any $\Lambda$ such that $|\Lambda|\geqslant \max\{3^{6d},(cd)^{cd^2}\}$, we have 
$$\pi_{\beta}\left(\begin{array}{c}
\exists x\in\Lambda\quad \sigma^+(x) = +1,\quad \sigma^D(x) = -1\\
\exists C \text{ a cut separating }x\text{ from }B\\
d(x,C)\geqslant \kappa c^2\ln^2|\Lambda|
\end{array}\right)\leqslant \frac{1}{|\Lambda|^c},$$
and 
$$\pi_{\beta}\left(\begin{array}{c}
\exists x\in\Lambda\quad \sigma^+(x) = -1,\quad \sigma^D(x) = +1\\
\exists C \text{ a cut separating }x\text{ from }T\\
d(x,C)\geqslant \kappa c^2\ln^2|\Lambda|
\end{array}\right)\leqslant \frac{1}{|\Lambda|^c}.
$$
\end{thm}
\noindent The key to obtain these two results is to construct a coupling $(X,Y)$, where $X$ is a standard FK-percolation configuration and $Y$ is a configuration where the top side $T$ and the bottom side $B$ of the box $\Lambda$ are disconnected. We denote this event by $\{T\nconnect B\}$. The localisation of the interface in the Ising model is induced by a control of the distance between the interface $\mathcal{I}$ and the set of the pivotal edges $\mathcal{P}$ of  the coupling $(X,Y)$. In this paper, we consider the FK-percolation model with a
parameter $p$ close to 1 and $q$ larger than $1$. Interfaces in a box $\Lambda$ are naturally created when
the configuration is conditioned to stay in the set $\{T\nconnect B\}$. The interface $\mathcal{I}$ is defined as
$$ \mathcal{I}= \big\{\,e\subset\Lambda : X(e) \neq Y(e)\,\big\} 
$$
and we denote by $\mathcal{P}$ the set of the pivotal edges for the event $\{T\nconnect B\}$ in $Y$. Our main result for the FK model is the following.
\begin{thm}\label{fk.main1}
For any $q\geqslant 1$, there exist $\tilde{p}< 1$ and $\kappa>0$, such that, for $p\geqslant \tilde{p}$, any $c\geqslant 1$ and any box $\Lambda$ such that $|\Lambda|\geqslant \max\{3^{6d},(cd)^{cd^2}\}$,
$$\mu_{\Lambda,p,q} \Big(\exists e\in \mathcal{P}\cup\mathcal{I},d\left(e,\Lambda^c\cup\mathcal{P}\setminus\{e\}\right)
	\geqslant \kappa c^2\ln^2 |\Lambda| \Big)\leqslant \frac{1}{|\Lambda|^c}.
$$
\end{thm}

\noindent Let us explain briefly how we build the measures $\mu_{\Lambda,p,q}$ and $\pi_{\Lambda,\beta}$ as well as the strategy for proving theorem \ref{fk.main2}. With the help of a Gibbs sampler algorithm (see section 8.4 of~\cite{FKgrimmett}), we construct a coupling between two Markov chains $(X_t,Y_t)_{t\in \mathbb{N}}$ on the space of the percolation configurations in a box $\Lambda$. The measure $\mu_{\Lambda,p,q}$ is the unique invariant measure of the process $(X_t,Y_t)_{t\in\mathbb{N}}$. Starting from a coupled configuration $(\omega,\omega')$ under the measure $\mu_{\Lambda,p,q}$, we put spins on the vertices in the box $\Lambda$ using an adaptation of the Edwards-Sokal coupling. By construction, the configuration $\omega$ dominates $\omega'$. We put spins at first on the vertices according to the configuration $\omega'$, under the Dobrushin boundary condition, to obtain a spin configuration $\sigma^D$. Then, we put spins according to $\omega$, under the plus (respectively minus) boundary condition to obtain the configuration $\sigma^+$ (respectively $\sigma^-$) with the restriction that an open cluster in $\omega$ also appearing in $\omega'$ has the same spin as in $\sigma^D$.  The measure $\pi_{\Lambda,\beta}$ is the probability distribution of $(\sigma^+,\sigma^-,\sigma^D)$ obtained from $\mu_{\Lambda,p,q}$ and the colouring. Each of its marginals is an Ising measure in the box with the corresponding boundary conditions.

As for the proof of theorem \ref{fk.main1}, we follow the ideas presented in \cite{new}. We control the distance between two pivotal edges by identifying a cut and a closed path disjoint from the cut. However, due to the correlations between all the edges in the FK-percolation model, we cannot use the BK inequality which holds for a product space and which is a key ingredient in $\cite{new}$. To solve this difficulty, we explore adequately the open clusters and we identify a sub-graph in $\Lambda$ containing a long closed path and outside of which we can find a cut. The configurations in this sub-graph can be compared to a Bernoulli configuration. To study the case where the distance between an edge of the interface and the pivotal edges is big, we show that the interface edge cannot have been created a long time ago. Moreover, at the time when it is created, it must be a pivotal edge. Therefore, the set of the pivotal edges must move rather fast. We obtain a control over the speed of the pivotal edges. This estimate relies on the study of specific space-time paths, which describe how the cut sets move. 

These results answer the question 4 raised in \cite{new}, and also give some information to the subsequent question 5. However, we would like to obtain more information about the structure of the set $\mathcal{P}_I$. 

This paper is organised as follows. In section 2, we give the definitions of the objects and the notations which we will use in this article. In section 3, we show the estimate on the distance between two pivotal edges in the FK-percolation model. In section 4, we control the speed of the pivotal edges. In section 5, we show theorem \ref{fk.main1} and we prove theorem \ref{fk.main2} (resp. \ref{fk.main3}) in section 6 (resp. 7).
\paragraph{Acknowledgments.} I warmly thank Jean-Baptiste Gouéré for his attentive reading
and for numerous constructive comments which were essential to improve the presentation of the results and the clarity of the proofs. 

\section{The notations}
In this section, we present the FK-percolation model which we study and we recall some fundamental tools which we will use in the rest of this paper. 
\subsection{Geometric definitions}
We start with some geometric definitions.
\paragraph{The lattice $\mathbb{L}^d$.} For an integer $d\geqslant 2$, the lattice $\mathbb{L}^d$ is the graph $(\mathbb{Z}^d,\mathbb{E}^d)$, where the set $\mathbb{E}^d$ is the set of pairs $\langle x, y\rangle$ of points in $\mathbb{Z}^d$ which are at Euclidean distance $1$.
\paragraph{The usual paths.} We say that two edges $e$ and $f$ are neighbours if they have one endpoint in common. A usual path is a sequence of edges $(e_i)_{1\leqslant i \leqslant n}$ such that for $1\leqslant i< n$, $e_i$ and $e_{i+1}$ are neighbours.
\paragraph{The $*$-paths.} In order to study the cuts in any dimension $d \geqslant 2$, we use $*$-connectedness on the edges as in~\cite{Pisztora1996Surface}. We consider the supremum norm on $\mathbb{R}^d$:
$$\forall x = (x_1,\dots,x_d)\in \mathbb{R}^d\qquad\parallel x\parallel_\infty = \max_{i = 1,\dots,d}|x_i|.
$$
For $e$ an edge in $\mathbb{E}^d$, we denote by $m_e$ the center of the unit segment associated to $e$. We say that two edges $e$ and $f$ of $\mathbb{E}^d$ are $*$-neighbours if $\parallel m_e-m_f\parallel_\infty\leqslant 1$. A $*$-path is a sequence of edges $(e_1,\dots,e_n)$ such that, for $1\leqslant i <n$, the edge $e_i$ and $e_{i+1}$ are $*$-neighbours. For a path $\gamma$, we denote by $\mathrm{support}(\gamma)$ the set of the edges of $\gamma$. We say that a path is simple if the cardinal of its support is equal to its length. 

\paragraph{The box $\Lambda$.} We will mostly work in a box $\Lambda$ centred at the origin (not necessarily straight) as illustrated in the figure \ref{fk.fig:box}. More precisely, we will consider a $d$-cube $\Lambda$ centred at origin. We can also consider $\Lambda$ as the graph $\Lambda = (V,E)$ is the sub-graph of $\mathbb{L}^d$ whose vertices are included in the cube. The boundary of $\Lambda$, denoted by $\partial\Lambda$, is defined as,
$$\partial \Lambda = \oset x\in V\, :\, \exists y \notin K, \langle x, y \rangle \in \mathbb{E}^d\cset.
$$
\begin{figure}[ht]
\centering{
\resizebox{100mm}{!}{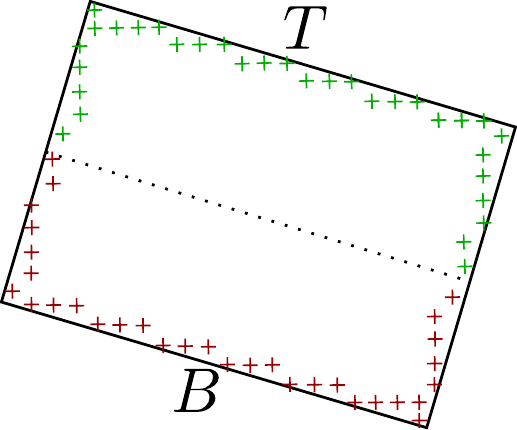}
\caption{The box $\Lambda$ and its boundary (the crosses). The green crosses form the side $T$ and the red ones form the side $B$.}
\label{fk.fig:box}
}
\end{figure}
We will distinguish two disjoint non-empty subsets of $\partial\Lambda$, denoted by $T$ and $B$. We consider a $(d-1)$ dimensional plane containing the origin and parallel to a side of $K$. This plane separates $\Lambda$ into two parts $\Lambda^+$ and $\Lambda^-$. The set $T$ is the subset of $\partial\Lambda$ included in $\Lambda^+$ and $B$ the one included in $\Lambda^-$.
\paragraph{The separating sets.} Let $A,B$ be two subsets of $\Lambda$. We say that a set of edges $S\subset \Lambda$ separates $A$ and $B$ if no connected subset of $\Lambda\cap\mathbb{E}^d\setminus S$ intersects both $A$ and $B$. Such a set $S$ is called a separating set for $A$ and $B$. We say that a separating set is minimal if there does not exist a strict subset of $S$ which separates $A$ and $B$.

\paragraph{The cuts.} We say that $S$ is a cut if $S$ separates $T$ and $B$, and $S$ is minimal for the inclusion.

\subsection{The Ising model}
Let $\Lambda= (V,E)$ be the finite box. We associate to each vertex $x\in V$ a random spin $\sigma(x)$ which can either be $+1$ or $-1$. The spin values are chosen according to a certain probability measure $\lambda_\beta$, known as a Gibbs state, which depends on a parameter $\beta\in [0,+\infty[$, and is given by
$$\lambda_\beta(\sigma)= \frac{e^{-\beta H(\sigma)}}{Z_I},\quad \sigma\in\{+1,-1\}^V,
$$
where $$H(\sigma) = -\sum_{\langle x,y\rangle \in E}\sigma(x)\sigma(y)$$ is the Hamiltonian and $Z_I$ is the normalisation constant called the partition function.
\subsection{The FK-percolation model}
Also known as the random-cluster model, the FK-percolation model is a generalisation of the Bernoulli percolation model, in which we introduce correlations between edges by taking into account the number of open clusters in a configuration. On a finite graph $(V,E)$, a random cluster-measure is a member of a certain class of probability measures on the space set $\{0,1\}^E$. Let $\omega$ belongs to $\{0,1\}^E$, we say that an edge $e$ is open if $\omega(e) = 1$ and closed if $\omega(e) = 0$, and we set $$\eta(\omega) = \oset e\in E\,:\, \omega(e) =1\cset.$$
Let $k(\omega)$ be the number of connected components (or the open clusters) of the graph $(V,\eta(\omega))$, and note that $k(\omega)$ includes the count of the isolated vertices, that is, of vertices incident to no open edge. For two parameters $p\in[0,1]$ and $q>0$, the \textit{random-cluster measure} $\Phi_{\Lambda,p,q}$ is defined as 
$$\fk{}{\Lambda,p,q}{\omega} =  \frac{1}{Z_{RC}}\left\lbrace \prod_{e\in E}p^{\omega(e)}(1-p)^{1-\omega(e)}\right\rbrace q^{k(\omega)}, \quad \omega \in \{0,1\}^E,
$$
where the partition function $Z_{RC}$ is given by 
$$Z_{RC} = \sum_{\omega\in \{0,1\}^E}\left\lbrace\prod_{e\in E}p^{\omega(e)}(1-p)^{1-\omega(e)}\right\rbrace q^{k(\omega)}.
$$
In our study, we will consider only the case where $q\geqslant
1$.
\paragraph{Boundary conditions.}
We will consider different boundary conditions. Let $\xi\in \{0,1\}^{\mathbb{E}^d}$ and $\Lambda = (V,E)$ be the box. Let $\Omega^\xi$ denote the subset of $\{0,1\}^{\mathbb{E}^d}$ consisting of all the configurations $\omega$ satisfying $\omega(e) = \xi(e)$ for $e\in\mathbb{E}^d\setminus E$. We shall write $\Phi_{\Lambda,p,q}^\xi$ for the random-cluster measure on $\Lambda$ with boundary condition $\xi$, given by
$$\fk{\xi}{\Lambda,p,q}{\omega} = \left\lbrace\begin{array}{ll}
\frac{1}{Z^\xi}\left\lbrace\prod_{e\in E}p^{\omega(e)}(1-p)^{1-\omega(e)}\right\rbrace q^{k(\omega,\Lambda)} & \text{if }\omega\in \Omega^\xi,\\
0& \text{otherwise},
\end{array} \right.
$$
where $k(\omega,\Lambda)$ is the number of components of the graph $(\mathbb{Z}^d,\eta(\omega))$ that intersect $\Lambda$, and $Z^\xi$ is the appropriate normalizing constant.
In particular, we will consider three special boundary conditions in this paper:
\begin{itemize}[leftmargin = 0.4cm]
\item The $0$-boundary condition corresponds to the case where all the edges of $\xi$ are closed. This condition is also called the \textit{free} boundary condition.
\item The $1$-boundary condition corresponds to the case where all the edges of $\xi$ are open. We can also see this condition as adding one vertex which is connected to all the vertices of $\partial\Lambda$, therefore, this boundary condition is called the \textit{wired} boundary condition.
\item The $TB$-boundary condition corresponds to the Dobrushin boundary con\-dition for the Ising model introduced in \cite{D72Gibbs}. This boundary condition corresponds to a configuration $\xi$ where all the vertices of $T$ are connected and all the vertices of $B$ are connected by open paths of edges outside of $\Lambda$, but there is no open path which connects a vertex of $T$ to a vertex of $B$. We can also see this condition as adding two vertices to the graph $\Lambda$, one of which is connected to all the vertices of $T$ and the other one connected to all the vertices of $B$. 
\end{itemize}
In order to simplify the notations, we omit $\Lambda,p,q,\xi$ in $\Phi_{\Lambda,p,q}^\xi$ if it doesn't create confusions. We will use two fundamental properties of the FK-perco\-lation model called the spatial Markov property (see chapter 4.2 of \cite{FKgrimmett}) and the comparison between different values of $p$ and $q$ stated in chapter 3.4 of~\cite{FKgrimmett}.

\subsection{Coupled dynamics of FK-percolation}
We will use a special Glauber process called the Gibbs sampler to study the conditioned FK-measure. Consider the finite graph $\Lambda = (V,E)$. To simplify the notations, we omit the $\Lambda,p,q$ in $\Phi_{\Lambda,p,q}$ in this section. The Gibbs sampler is a Markov chain $(X_t)_{t\in\mathbb{N}}$ on the state space $\Omega = \{0,1\}^E$. At a time $t$, we choose an edge $e$ uniformly in $E$, and we set the status of $e$ according to the current states of the other edges. More precisely, let $(U_t)_{t\in \mathbb{N}}$ be a sequence of uniform variables on $[0,1]$ and $(E_t)_{t\in N}$ be a sequence of uniform variables in $E$. For $\omega\in\Omega$ and $e\in E$, we denote by $\omega^e$ the configuration obtained by opening the edge $e$ and by $\omega_e$ the one where $e$ is closed. At time $t$, we suppose $X_{t-1} = \omega$ and we set 
$$X_{t}(e) = \left\lbrace\begin{array}{cl}
\omega(e) & \text{if } E_t \neq e\\
1 & \text{if } E_t = e \text{ and }U_t\geqslant \frac{\Phi({\omega}_e)}{\Phi({\omega}^e)+\Phi({\omega}_e)}\\
0 & \text{if } E_t = e \text{ and }U_t< \frac{\Phi({\omega}_e)}{\Phi({\omega}^e)+\Phi({\omega}_e)}
\end{array}\right..
$$
We also define a coupled process $(Y_t)_{t\in \mathbb{N}}$ which stays in $\oset T\nconnect B\cset$. We use the same sequences $(E_t)_{t\in \mathbb{N}}$ and $(U_t)_{t\in\mathbb{N}}$. At time $t$, we suppose $Y_{t-1}(e) = \omega$ and we change the status of the edge $e$ in $Y_t$ as follows:
$$Y_{t}(e) = \left\lbrace\begin{array}{cl}
\omega(e) & \text{if } E_t \neq e\\
1 & \text{if } E_t = e, U_t\geqslant \frac{\Phi({\omega}_e)}{\Phi({\omega}^e)+\Phi(\omega_e)} \text{ and }T\overset{\omega^e}{\nconnect} B\\
0 & \text{if } E_t = e, U_t\geqslant \frac{\Phi(\omega_e)}{\Phi(\omega^e)+\Phi(\omega_e)} \text{ and }T\overset{\omega^e}{\longleftrightarrow} B\\
0 & \text{if } E_t = e,U_t< \frac{\Phi(\omega_e)}{\Phi(\omega^e)+\Phi(\omega_e)}
\end{array}\right..
$$
Before opening a closed edge $e$ at time $t$, we verify whether this will create a connexion between $T$ and $B$ in $Y_t$. If it is the case, the edge $e$ stays closed in $Y_t$ but can be opened in $X_t$, otherwise the edge $e$ is opened in both $X_t$ and $Y_t$. On the contrary, the two processes behave similarly for the edge closing events since we cannot create a new connexion by closing an edge. The set of the configurations satisfying $\oset T\nconnect B \cset$ is irreducible and the process $(X_t)_{t\in\mathbb{N}}$ is reversible. By the lemma 1.9 of~\cite{Kelly:2011:RSN:2025239}, there exists a unique stationary distribution $\Phi^\mathcal{D}$ for the process $(Y_t)_{t\in\mathbb{N}}$ and $\Phi^\mathcal{D}$ is equal to the probability $\Phi^{TB}$ conditioned by the event $\oset T\nconnect B\cset$, i.e., 
$$\Phi^\mathcal{D} (\cdot)= \Phi^{TB}(\cdot \,|\, T\nconnect B) .
$$
Suppose that we start from a configuration $(X_0,Y_0)$ belonging to the set 
$$\mathcal{E} = \oset (\omega_1,\omega_2)\in \{0,1\}^{\mathbb{E}^d\cap \Lambda}\times \{T\nconnect B\}\,:\, \forall e\subset \Lambda\quad \omega_1(e)\geqslant \omega_2(e)\cset.
$$
The set $\mathcal{E}$ is irreducible and aperiodic. In fact, each configuration of $\mathcal{E}$ communicates with the configuration where all edges are closed. The state space $\mathcal{E}$ is finite, therefore the Markov chain $(X_t,Y_t)_{t\in\mathbb{N}}$ admits a unique equilibrium distribution $\mu_p$. We denote by $P_\mu$ the law of the process $(X_t,Y_t)_{t\in\mathbb{N}}$ starting from a random initial configuration $(X_0,Y_0)$ with distribution $\mu_{\Lambda,p,q}$.
We define the following objects using the previous coupling.
\begin{defi}The interface at time $t$ between $T$ and $B$, denoted by $\mathcal{I}_t$, is the set of the edges in $\Lambda$ that differ in the configurations $X_t$ and $Y_t$, i.e.,
$$ \mathcal{I}_t = \big\{\,e\subset\Lambda : X_t(e) \neq Y_t(e)\,\big\}.
$$
\end{defi}
\noindent The edges of $\mathcal{I}_t$ are open in $X_t$ but closed in $Y_t$ and the configuration $X_t$ is above the configuration  $Y_t$. We define next the set $\mathcal{P}_t$ of the pivotal edges for the event $\{T\nconnect B\}$ in the configuration $Y_t$.
\begin{defi}
The set $\mathcal{P}_t$ of the pivotal edges in $Y_t$ is the collection of the edges in $\Lambda$ whose opening would create a connection between $T$ and $B$, i.e.,
$$\mathcal{P}_t = \big\{ \,e\subset\Lambda : T\overset{Y_t^e}{\longleftrightarrow} B\,\big\}.
$$
\end{defi}
\noindent We define finally the set $\mathcal{C}_t$ of the cuts in $Y_t$.
\begin{defi}
The set $\mathcal{C}_t$ of the cuts in $Y_t$ is the collection of the cuts in $ \Lambda$ at time $t$.
\end{defi}
\subsection{The classical Edwards-Sokal coupling}
We wish to gain insight into the interface in the Ising model with the help of our previous results and the classical coupling of Edwards and Sokal (see chapter 1 of \cite{FKgrimmett} for more details on this coupling). Let $\Lambda = (V,E)$ be the box. We consider the product space $\Sigma\times \Omega$ where $\Sigma = \{-1,1\}^V$ and $\Omega = \{0,1\}^E$. We define a probability $\nu$ on $\Sigma\times \Omega$ by 
$$\nu(\sigma,\omega) \propto \prod_e\Big\{ (1-p)\delta_{\omega(e),0}+ p\delta_{\omega(e),1}\delta_e(\sigma) \Big\}, \quad (\sigma,\omega)\in \Sigma\times \Omega,
$$
where $\delta_e(\sigma) = \delta_{\sigma(x),\sigma(y)}$ for $e =\langle x,y\rangle\in E$. The constant of proportionality is the one which ensures the normalization 
$$\sum_{(\sigma,\omega)\in \Sigma\times \Omega} \nu(\sigma,\omega) = 1.
$$
For $q = 2$, $p = 1-e^{-\beta}$ and $\omega\in\Omega$, the conditional measure $\nu(\cdot|\omega)$ on $\Sigma$ is obtained by colouring randomly the clusters of $\omega$. More precisely, conditionally on a percolation configuration $\omega$, the spins are constant on the clusters of $\omega$ and they are independent between the clusters. With the help of this coupling, we can transport results in FK-percolation to the Ising model.
\subsection{The coupling of spin configurations}
We construct a coupling of the Ising configurations $(\sigma^+,\sigma^-,\sigma^D)$ with different boundary conditions from a pair of percolation configurations $$(\omega,\omega')\in \Omega\times\{T\nconnect B\}$$ satisfying $\omega\geqslant\omega'$. The configuration $\sigma^+$ (resp. $\sigma^-$) will correspond to the spin configuration with the $+$ boundary condition (resp. $-$ boundary condition) and the configuration $\sigma^D$ will correspond to the Dobrushin boundary condition. We will put spins on the vertices in $\Lambda$ as follows. We start by putting spins on the vertices of the clusters of $\omega'$ using the Edwards-Sokal coupling with the Dobrushin boundary conditions. This way we obtain a spin configuration, which we denote by $\sigma^D$. Notice that the open clusters of $\omega$ are unions of the open clusters of $\omega'$. For an open cluster in $\omega$ which touches the boundary of $\Lambda$, we color its vertices with $+1$ in the configuration $\sigma^+$ and with $-1$ in $\sigma^-$. For an open cluster $C$ which does not touch the boundary of $\Lambda$, if $C$ is also an open cluster in $\omega'$, we set
$$\forall x\in C\quad \sigma^+(x) = \sigma^-(x) = \sigma^D(x).$$
For an open cluster $C$ of $\omega$ which does not touch the boundary of $\Lambda$ and which is the union of several open clusters of $\omega'$, we set 
$$\forall x\in C\quad \sigma^+(x) = \sigma^-(x) = N(C) ,
$$
where $N(C)$ is equal to $1$ with probability $1/2$ and $-1$ with probability $1/2$. Of course, the random variables $N(C)$ associated to the clusters of $\omega$ are independent and also independent from the configuration outside $C$. For $q = 2, p= 1-e^{-\beta}$ and $(\omega,\omega')$ distributed under $\mu_{\Lambda,p,2}$, the configurations $\sigma^+,\sigma^-$ and $\sigma^D$ obtained are distributed according to the Gibbs state at inverse temperature $\beta$ with boundary conditions $+,-$ and Dobrushin. We denote by $\pi_{\Lambda,\beta}$ the distribution of the triple $(\sigma^+,\sigma^-,\sigma^D)$.
\section{Localising a cut around the pivotal edges}
The following proposition controls the distance between an edge belonging to a cut and the set of the pivotal edges. 
\begin{prop} \label{fk.distpivot}
For any $q\geqslant 1$, there exist $\tilde{p}<1$ and $\kappa>1$ such that, for $p\geqslant \tilde{p}$, and for any $c \geqslant 1$ and any box $\Lambda$ satisfying $|\Lambda|> 3^{6d}$, we have
$$ \cfk{TB}{\Lambda,p,q}{\exists C\in\mathcal{C},\exists e\in C, d(e,\mathcal{P}\cup\Lambda^c\setminus \{e\})\geqslant \kappa c\ln|\Lambda|}{T\nconnect B} \leqslant \frac{1}{|\Lambda|^c}.
$$
\end{prop}
\noindent The strategy of the proof follows that of proposition 1.4 of~\cite{new}. We can still observe a closed path which is disjoint from a cut. However, one key ingredient of the proof in~\cite{new} is the BK inequality (see \cite{grimmett1999percolation}) which doesn't hold for the FK-percolation model. The following lemma gives us an inequality which plays the role of the BK inequality in the proof.

\begin{lem}\label{fk.fbk}
\noindent Let $n\geqslant 2$, $q\geqslant 1$, $p\in[0,1]$, and let $e$ be an edge in $\Lambda$. We define the event
$$\Gamma(e,n,T) =\left\lbrace
\begin{array}{c}\text{there exists a closed path of length }n\text{ starting from }e \\\text{ and there exists a cut which separates this path from }T\end{array}\right\rbrace.
$$
The following inequality holds:
$$\fk{TB}{\Lambda,p,q}{\Gamma(e,n,T)}\\
\leqslant \alpha^n(d)\left(1-f(p,q)\right)^n\fk{TB}{\Lambda}{T\nconnect B},
$$
where $$
f(p,q)=\frac{p}{p+q-pq}
$$
and $\alpha(d)$ is the number of $*$-neighbours of an edge in the lattice $\mathbb{L}^d$.
\end{lem}
\noindent Notice that this inequality also holds if we consider the symmetric event $\Gamma(e,n,B)$.
\begin{proof}
Let us start with the construction of a random graph in $\Lambda$ in which we can find a path starting from $e$ and the complementary of which contains a cut. The construction is inspired by the proof of theorem 5.3 in \cite{MR1379156}. For the topological complications in the construction, we refer to the related passage of \cite[Sect. 2]{MR876084}. We define the open cluster of $T$ as
$$O(T) = \left\lbrace x\in \Lambda\,:\,x\longleftrightarrow T \right\rbrace.
$$
We consider a configuration satisfying $\{T\nconnect B\}$.
Then we define the set $C^+$ as the set of the edges $f$ satisfying
\begin{itemize}[leftmargin = 0.6cm]
\item[1.] $f$ has exactly one endpoint in $O(T)$;
\item[2.] there exists a geometric path in $\mathbb{L}^d$ which connects $B$ to $f$ and which does not use vertices of $O(T)$.
\end{itemize}
Notice that the path in point 2 is only geometric, there is no requirement on the status of the edges of this path.
We note three facts about the set $C^+$:
\begin{itemize}[leftmargin = 0.6cm]
\item[a.] the edges of $C^+$ are closed,
\item[b.] the set $C^+$ contains a cut,
\item[c.] $C^+$ is measurable with respect to the configuration of the edges which have at least one endpoint in $O(T)$.
\end{itemize}
Let us now consider a configuration in the event $\Gamma(e,n,T)$.
The cut in $C^+$ separates $e$ from $T$. We now construct a sub-graph $\mathcal{G}$ of $\Lambda$.  We define $V'$ as the set of the vertices included in $V\setminus O(T)$ which are connected to $B$ without using an edge of $C^+$. We define  the sub-graph $\mathcal{G}$ as 
$$\mathcal{G} = (V',\{\langle x,y\rangle\,:\, x,y\in V'\}).
$$The figure~\ref{fk.fig:sponge} illustrates the construction of the graph $\mathcal{G}$.
\begin{figure}[ht]
\centering{
\resizebox{100mm}{!}{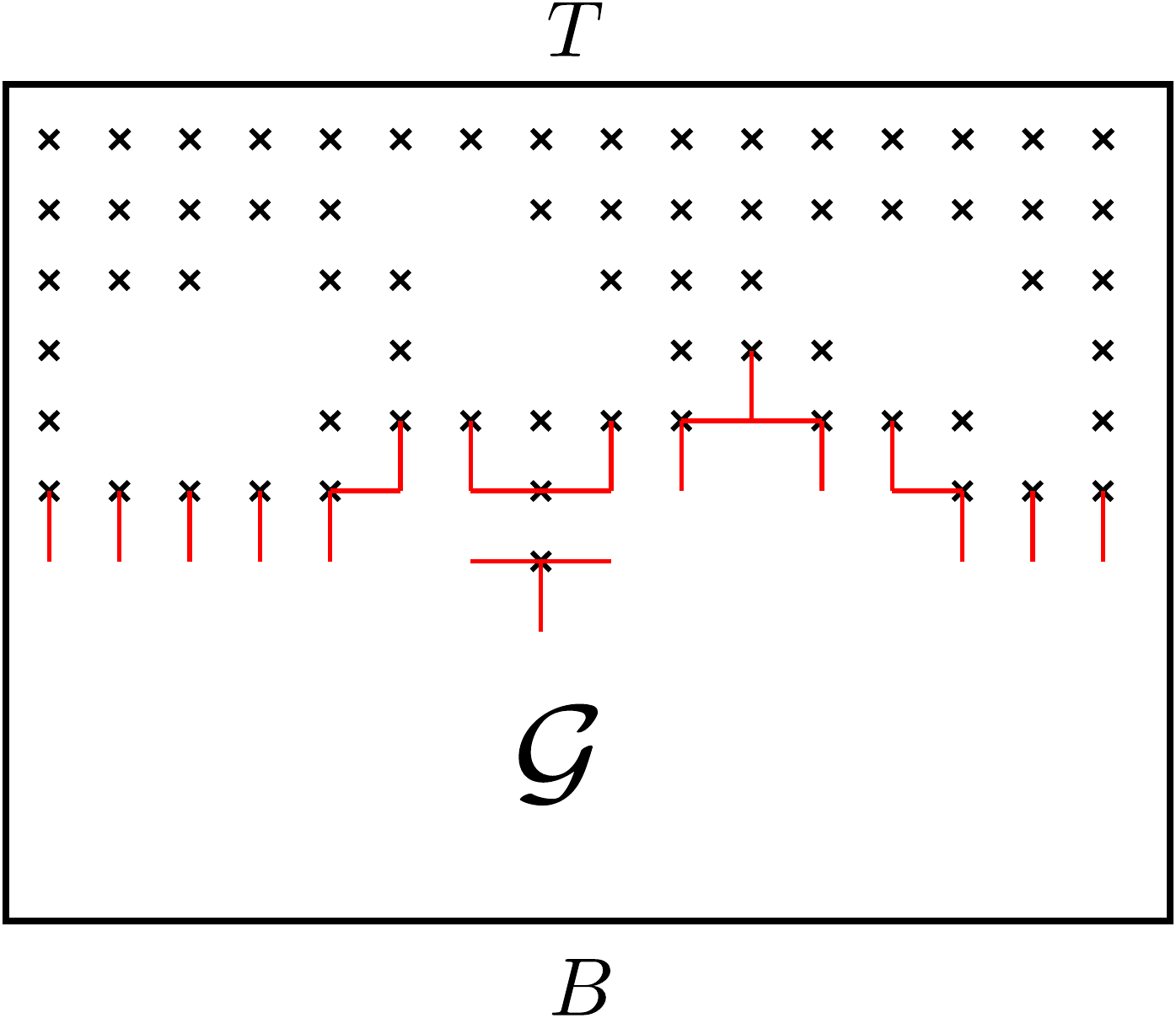}
\caption{The crosses are the vertices of $O(T)$ and the red edges form the set $C^+$, the graph $\mathcal{G}$ remains unexplored.}
\label{fk.fig:sponge}
}
\end{figure} 
From the construction, it follows that the cut in $C^+$ does not contain an edge of the graph $\mathcal{G}$. For a fixed sub-graph $G_1 = (V_1,E_1)$ of $(V,E)$, we denote by $G'$ the graph $(V',E')$, where $E' = E\setminus E_1$ and $V'$ is the set of the endpoints of the edges in $E'$. We decompose the set of the edge configurations $\Omega$ in $\Lambda$ as $$\Omega = \Omega_{G_1}\times \Omega_{G'},$$ where $\Omega_G$ (respectively $\Omega_{G'}$) is the set of the configurations of the edges in $G_1$ (respectively $G'$). We notice that the event $\mathcal{G} = G_1$ is entirely determined by the configurations $ \omega_{G'}\in \Omega_{G'}$. We obtain that 
\begin{multline}\fk{TB}{\Lambda,p,q}{\Gamma(e,n,T)} =\sum_{G_1}\fk{TB}{\Lambda,p,q}{\Gamma(e,n,T),\mathcal{G} = G_1} \\= \sum_{G_1}\sum_{\eta\in \Omega_{G'}: \mathcal{G} = G_1}\fk{TB}{\Lambda,p,q}{\Gamma(e,n,T),\omega_{G'} = \eta}\\
= \sum_{G_1}\sum_{\eta\in \Omega_{G'}: \mathcal{G} = G_1}\cfk{TB}{\Lambda,p,q}{\Gamma(e,n,T)}{\omega_{G'} = \eta}\fk{TB}{\Lambda,p,q}{\omega_{G'}=\eta}.\label{fk.sum_mkv}
\end{multline}
We define the event $\Gamma(e,n)$ as
$$\Gamma(e,n) =\oset \text{there exists a closed $*$-path of length }n\text{ starting from }e\cset.
$$
By the spatial Markov property, for a fixed graph $G_1$ containing all the edges intersecting $B$ and $\eta$ such that $\mathcal{G} = G_1$, we have
$$\cfk{TB}{\Lambda,p,q}{\Gamma(e,n,T)}{\omega_{G'} = \eta} = \fk{\xi(\eta)}{G_1,p,q}{\Gamma(e,n)}\leqslant\fk{0}{G_1,p,q}{\Gamma(e,n)},
$$
where $\xi(\eta)$ is the boundary condition on $G_1$ induced by $\eta$. The last inequality holds because the event $\Gamma(e,n)$ is decreasing.
As for the configurations in the graph $G_1$, we can compare the measure $\Phi_{G_1,p,q}^{0}$ with the Bernoulli percolation measure with parameter $$
f(p,q) = \frac{p}{p+q-pq}.$$
Since the event $\Gamma(e,n)$ is decreasing, by the comparison inequalities stated in chapter 3.4 of~\cite{FKgrimmett}, we have
$$\fk{0}{G_1,p,q}{\Gamma(e,n)}\leqslant \fk{0}{G_1,f(p,q),1}{\Gamma(e,n)}.
$$
By a standard Peierls argument, we have then
$$\fk{0}{G_1,f(p,q),1}{\Gamma(e,n)}\leqslant \alpha^n(d)(1-f(p,q))^n,
$$
where $\alpha(d)$ is the number of $*$-neighbours of an edge in dimension $d$.
We use this upper bound in the sum~\eqref{fk.sum_mkv} and we obtain
\begin{multline*}\fk{TB}{\Lambda}{\Gamma(e,n,T)}\leqslant \alpha^n(d)(1-f(p,q))^n\sum_{G_1:B\subset V(G_1)}\sum_{\eta:\mathcal{G}=G_1}\fk{TB}{\Lambda}{\omega_{G'} = \eta}\\
= \alpha^n(d)(1-f(p,q))^n\sum_{G_1:B\subset V(G_1)}\fk{TB}{\Lambda}{\mathcal{G} = G_1}.
\end{multline*}
We now calculate the last sum. We have
\begin{multline*}\sum_{G_1:B\subset V(G_1)}\fk{TB}{\Lambda}{\mathcal{G} = G_1}=
\mfk{TB}{\Lambda}{\begin{array}{c}\exists \mathcal{G} \text{ sub-graph of }\Lambda\\
V(\mathcal{G})\cap O(T) = \emptyset,\quad B\subset V(\mathcal{G})\end{array}}\\ = \fk{TB}{\Lambda}{O(T)\cap B = \emptyset}.
\end{multline*}
The event $\{O(T)\cap B = \emptyset\}$ implies the disconnection between $T$ and $B$, thus we have
$$\fk{TB}{\Lambda}{\Gamma(e,n,T)}\\
\leqslant \alpha^n(d)\left(1-f(p,q)\right)^n\fk{TB}{\Lambda}{T\nconnect B},
$$
which is the desired inequality.
\end{proof}
We now prove proposition~\ref{fk.distpivot} with the help of the previous lemma. 
\begin{proof}[Proof of proposition~\ref{fk.distpivot}]
Notice that every edge of a cut is connected to the set of the pivotal edges $\mathcal{P}$ or to the boundary of $\Lambda$ by a closed path. Let us fix an edge $e$ which belongs to a cut and which is at distance more than $\kappa c \ln|\Lambda|$ from $\mathcal{P}\cup\Lambda^c\setminus \{e\}$. There is a closed path starting from the edge $e$ and which is of length $\kappa c\ln|\Lambda|/4d$. This path is disjoint from a cut since there is no pivotal edge on this path. We refer to the section 3 of~\cite{new} for the detailed geometric justifications. The probability appearing in the proposition is less than
$$ \cfk{TB}{\Lambda}{A}{T\nconnect B},
$$ 
where $A$ is the event
$$A = \left\lbrace\begin{array}{c} \exists \gamma \text{ closed path of length }\frac{\kappa c\ln|\Lambda|}{4d} \text{ starting from }e\\
\exists C\in \mathcal{C}\quad C \text{ disjoint from }\gamma
\end{array}\right\rbrace.
$$
Since the existence of a cut implies the disconnection between $T$ and $B$, we can rewrite the conditioned probability as 
$$\cfk{TB}{\Lambda}{A}{T\nconnect B} = \frac{\fk{TB}{\Lambda}{A}}{\fk{TB}{\Lambda}{T\nconnect B}}.
$$
We distinguish two cases according to the positions of the path $\gamma$ and the cut $C$. Since the cut $C$ splits the box $\Lambda$ into two parts and the path $\gamma$ starting from $e$ is disjoint from the cut, then it is included in one of the two parts. Therefore, the path $\gamma$ is either separated from $T$ or from $B$ by $C$. We obtain
$$\fk{TB}{\Lambda}{A}\leqslant \fk{TB}{\Lambda}{\Gamma(e,T)}+\fk{TB}{\Lambda}{\Gamma(e,B)},
$$
where we reuse the notation $\Gamma$ introduced in the previous lemma by setting $n = \kappa c \ln|\Lambda|/4d$ and we omit $n$ in the notation as follows:
$$\Gamma(e,T) = \left\lbrace\begin{array}{c}\exists \gamma \text{ closed path of length }\frac{\kappa c\ln|\Lambda|}{4d} \text{ starting from }e\\
\exists C\in \mathcal{C}, C \text{ separates }\gamma\text{ from }T\end{array}\right\rbrace$$ and $\Gamma(e,B)$ is defined similarly, with $B$ instead of $T$.
By lemma \ref{fk.fbk}, we have 
$$\fk{TB}{\Lambda}{\Gamma(e,T)}\leqslant \big(\alpha(d)(1-f(p,q))\big)^{\kappa c\ln|\Lambda|/4d}\fk{TB}{\Lambda}{T\nconnect B}
$$
and the same holds for $\fk{TB}{\Lambda}{\Gamma(e,B)}$.
Thus, we have 
$$\frac{\fk{TB}{\Lambda}{A\circ\{\exists C\in\mathcal{C}\}}}{\fk{TB}{\Lambda}{T\nconnect B}}\leqslant 2\big(\alpha(d)(1-f(p,q))\big)^{\kappa c\ln|\Lambda|/4d}.
$$
For $q\geqslant 1$, there exist a $\tilde{p}<1$ and $\kappa\geqslant 1$, such that, for $p\geqslant \tilde{p}$ and $c\geqslant 1$, we have
$$\big(\alpha(d)(1-f(p,q))\big)^{\kappa c\ln|\Lambda|/4d}\leqslant  \frac{1}{2|\Lambda|^c},
$$
which yields the desired result.
\end{proof}
\section{The speed estimate of the pivotal edges}
We now study the difference between a cut at two different times $t$ and $t+s$. We begin with a key lemma which controls the number of closing events which can be realised on an interval of length $s$.
\begin{lem}\label{fk.fkfermeture}
Let $\Gamma$ be a simple $*$-path and $t$ a time and $s\in\mathbb{N}$. For any $k\in\{0, \dots,|\mathrm{support}(\Gamma)|\}$ and any configuration $y$ such that exactly $k$ edges of $\Gamma$ are closed, we have the following inequality:
$$P_\mu\left(\kern-5pt\begin{array}{c|c}\Gamma\text{ closed in }Y_{t+s} &Y_t = y\kern-5pt\end{array}\right)\\ \leqslant 
\left(\frac{s(1-p)}{|\Lambda|(1+p/q-p)}\right)^{|\mathrm{support}(\Gamma)|-k}.
$$
\end{lem}
\begin{proof}
Let us denote by $n$ the cardinal of $\mathrm{support}(\Gamma)$. If an edge $e$ is visited, then it must be closed at time $t$ or become closed at a time after $t$. Therefore, there exist $n-k$ different edges $e_1,\dots,e_{n-k}$ of $\Gamma$, and $n-k$ different times $t_1,\dots,t_{n-k}$ strictly bigger than $t$ such that 
$$\forall i\in\{1,\dots,n-k\} \qquad E_{t_i} = e_i\quad \text{and}\quad U_{t_i}\leqslant \frac{1-p}{1-p+p/q}.$$ 
These events are independent of the configuration $Y_{t}$ and we obtain an upper bound on the desired probability as follows:
\begin{equation}P_\mu\left(\kern-5pt\begin{array}{c|c}\Gamma\text{ closed in }Y_{t+s} &Y_t = y\kern-5pt\end{array}\right)\\
 \leqslant P\left(\begin{array}{c}\forall i \in \{1,\dots,n-k\}\\
 \exists t_i\in ]t,t+s]\quad E_{t_i}=e_i\\
U_{t_i}\leqslant \frac{1-p}{1-p+p/q}
\end{array}\right).\label{fk.ineqstc}
\end{equation}
This last probability depends on the i.i.d. sequence of couples $(E_t,U_t)_{t\in N}$. Moreover, the sets
$$J(e) = \oset  s<j\leqslant t\,:\, E_j = e\cset,\quad  e\in \{e_1,\dots,e_{n-k}\}
$$
are pairwise disjoint subsets of $\mathbb{N}$. For an edge $e\in \{e_1,\dots,e_{n-k}\}$, the event 
$$\left\lbrace\begin{array}{c}\exists r\in ]s,t] \quad E_{r}=e\\
U_{r}\leqslant \frac{1-p}{1-p+p/q}
\end{array}\right\rbrace
$$
is entirely determined by $(E_t,U_t)_{t\in J(e)}$. Therefore, by the BK inequality (see \cite{grimmett1999percolation}) applied with the random variables $(E_r,U_r)_{s<r\leqslant t}$,
the second probability in~\eqref{fk.ineqstc} is less than
\begin{equation}\prod_{e\in\{e_1,\dots,e_{n-k}\}}P\left(\begin{array}{c}\exists r\in ]t,t+s]\quad E_{r}=e\\
U_{r}\leqslant \frac{1-p}{1-p+p/q}
\end{array}\right).\label{fk.pprode}
\end{equation}
For each $e\in\{e_1,\dots,e_{n-k}\}$, we have 
$$P\left(\begin{array}{c}\exists r\in ]t,t+s]\quad E_{r}=e\\
U_{r}\leqslant \frac{1-p}{1-p+p/q}
\end{array}\right)
\leqslant \frac{s(1-p)}{|\Lambda|(1-p+p/q)}.
$$
We conclude that \eqref{fk.pprode} is less or equal than
$$\left(\frac{s(1-p)}{|\Lambda|(1-p+p/q)}\right)^{n-k}.
$$
Combined with \eqref{fk.ineqstc}, we have the inequality stated in the lemma.
\end{proof}
In order to apply the previous result to control the speed of a cut, we will consider $s$ satisfying $0<s\leqslant |\Lambda|$ and we study the probability that there exists a pivotal edge $e$ at time $t+s$ which is at distance $\ell$ from the pivotal edges at time $t$. 
The following proposition gives an upper bound on the speed of a cut.
\begin{prop} \label{fk.vconst}
There exist $\tilde{p}<1$ and $\kappa \geqslant 1$, such that for $p\geqslant \tilde{p}$, $t\in \mathbb{N}$, $s\in \{1,\dots,|\Lambda|\}$, $\ell\geqslant 1$ and any edge $e\subset \Lambda$ at distance more than $\ell$ from $\Lambda^c$, 
$$
P_\mu\Big(e\in \mathcal{P}_{t+s} \,\big|\, \exists c_t \in \mathcal{C}_t,\, d(e,c_t)\geqslant\ell
\Big)\leqslant e^{-\ell}.
$$
\end{prop}
\noindent The proof follows the ideas in the proof of proposition~4.1 of~\cite{new}.
\begin{proof}
We start by rewriting the conditioned probability as
\begin{equation}\frac{P_\mu\Big(\{e\in \mathcal{P}_{t+s}\}\,\cap\,\{\exists c_t \in \mathcal{C}_t, d(e,c_t) \geqslant \ell \}\Big)}{P_\mu\Big(\exists c_t \in \mathcal{C}_t, d(e,c_t) \geqslant \ell\Big)}.\label{fk.quotien}
\end{equation}
Since a pivotal edge is in a cut, there exists a closed simple $*$-path $\gamma$ which connects $e$ to an edge intersecting the boundary of $\Lambda$ at time $t+s$. This $*$-path travels a distance at least $\ell$ and the length of $\gamma$ is at least $\ell/d$. Therefore, by taking the first $n$ edges such that the support of this sequence is equal to $\ell/2d$, the numerator is bounded from above by
\begin{multline}P_\mu\left(\kern-5pt\begin{array}{ccc}\left\lbrace\kern-5pt\begin{array}{c}\exists \gamma \text{ simple closed $*$-path of length}\\
 n \text{ starting at }e\text{ at time } t+s\end{array}\kern-5pt\right\rbrace\kern-5pt&\bigcap &\kern-5pt\left\lbrace\kern-5pt\begin{array}{c}
\exists c_t \in \mathcal{C}_t \\ d(e,c_t) \geqslant \ell
\kern-5pt\end{array}\right\rbrace\end{array}
\kern-5pt\right)\leqslant\\
\sum_{\Gamma}P_\mu\left(\kern-5pt\begin{array}{ccc}\left\lbrace\kern-5pt\begin{array}{c} \Gamma \text{ simple closed $*$-path of length}\\
 n \text{ starting at }e\text{ at time } t+s\end{array}\kern-5pt\right\rbrace\kern-5pt&\bigcap &\kern-5pt\left\lbrace\kern-5pt\begin{array}{c}
\exists c_t \in \mathcal{C}_t \\ d(e,c_t) \geqslant \ell
\kern-5pt\end{array}\right\rbrace\end{array}
\kern-5pt\right).\label{fk.pSTCcapCut}
\end{multline}
For a $*$-path $\Gamma$ of length $n$ starting from $e$, we consider the set $M(k)$ of the configurations defined as 
$$
M(k) = \left\lbrace \begin{array}{cl}\omega\,:&\begin{array}{l} \exists F\subset \mathrm{support}(\Gamma),\, |F| = k\\
\forall f \in F\quad \omega(f) = 0\\
\forall f \in \mathrm{support}(\Gamma)\setminus F\quad \omega(f) = 1\\
\exists C\in \mathcal{C} \quad d(e,C)\geqslant \ell\end{array}\end{array}\right\rbrace.
$$
For $k$ fixed, the probability of having exactly $k$ edges in $\Gamma$ which are closed at time $t$ is less than
$$\sum_{y\in M(k)}P_\mu\left( \begin{array}{c|c}\begin{array}{c}\Gamma \text{ simple $*$-path}\\
\text{ closed at time } t+s\end{array}& Y_t = y\end{array}\right)P_\mu\big(Y_t = y\big).
$$
By lemma~\ref{fk.fkfermeture}, each term of the sum is less than
$$\left(\frac{s(1-p)}{|\Lambda|(1+p/q-p)}\right)^{n-k}P_\mu\big(Y_t = y\big).
$$
We obtain an upper bound for~\eqref{fk.pSTCcapCut} as
$$\sum_{\Gamma}\sum_k\left(\frac{s(1-p)}{|\Lambda|(1+p/q-p)}\right)^{n-k}P_\mu\big(Y_t\in M(k)\big).
$$
Let us calculate the probability $P_\mu(Y_t\in M(k))$. Notice that this event is determined by the configuration $Y_t$. This probability is less than 
\begin{multline}\mfk{TB}{\Lambda}{\begin{array}{c|c}\begin{array}{c} \exists F\subset \mathrm{support}(\Gamma),\, |F| = k\\
\forall f \in F\quad f\text{ closed}\\
\exists C\in \mathcal{C} \quad d(e,C)\geqslant \ell\end{array}& T\nconnect B\end{array}} \\
= \frac{\mfk{TB}{\Lambda}{\begin{array}{c} \exists F\subset \mathrm{support}(\Gamma),\, |F| = k\\
\forall f \in F\quad f\text{ closed}\\
\exists C\in \mathcal{C} \quad d(e,C)\geqslant \ell\end{array}}}{\fk{TB}{\Lambda}{T\nconnect B}}.\label{fk.decomp}
\end{multline}
Denote by $\Delta$ the set of the edges at distance less than $\ell-1$ from $e$, then the event 
$$\left\lbrace\begin{array}{c} \exists F\subset \mathrm{support}(\Gamma),\, |F| = k\\
\forall f \in F\quad f\text{ closed}\end{array}\right\rbrace
$$
depends on the edges inside $\Delta$, whereas the event 
$$D = \oset \exists C\in \mathcal{C} \quad d(e,C)\geqslant \ell \cset
$$
depends on the edges in $\Lambda\setminus \Delta$. Denote by $\Omega_{\Lambda\setminus \Delta}$ the set of the configurations of the edges in $\Lambda\setminus \Delta$. By the spatial Markov property, we have
\begin{multline}\mfk{TB}{\Lambda}{\begin{array}{c} \exists F\subset \mathrm{support}(\Gamma),\, |F| = k\\
\forall f \in F\quad f\text{ closed}\\
\exists C\in \mathcal{C} \quad d(e,C)\geqslant \ell\end{array}} \\= \sum_{\omega\in \Omega_{\Lambda\setminus \Delta}\cap D}\mfk{TB}{\Lambda}{\begin{array}{c|c}\begin{array}{c} \exists F\subset \mathrm{support}(\Gamma),\, |F| = k\\
\forall f \in F\quad f\text{ closed}\end{array}& \,\omega\end{array}}\fk{TB}{\Lambda}{\omega}\\
 = \sum_{\omega\in \Omega_{\Lambda\setminus \Delta}\cap D}\mfk{\xi(\omega)}{\Delta}{\begin{array}{c} \exists F\subset \mathrm{support}(\Gamma),\, |F| = k\\
\forall f \in F\quad f\text{ closed}\end{array}}\fk{TB}{\Lambda}{\omega}.\label{fk.sumpMk}
\end{multline}
Since the event 
$$\left\lbrace\begin{array}{c} \exists F\subset \mathrm{support}(\Gamma),\, |F| = k\\
\forall f \in F\quad f\text{ closed}\end{array}\right\rbrace
$$
is decreasing, for any boundary condition $\xi(\omega)$ on $\Delta$ induced by $\omega$, we have 
$$\mfk{\pi(\omega)}{\Delta}{\begin{array}{c} \exists F\subset \mathrm{support}(\Gamma),\, |F| = k\\
\forall f \in F\quad f\text{ closed}\end{array}}\leqslant \mfk{0}{\Delta}{\begin{array}{c} \exists F\subset \mathrm{support}(\Gamma),\, |F| = k\\
\forall f \in F\quad f\text{ closed}\end{array}}.
$$
We compare this last probability to the one under Bernoulli percolation with parameter $$
f(p,q) = \frac{p}{p+q-pq}.
$$
By the comparison inequalities between different values of $p,q$, we have 
$$\mfk{0}{\Delta}{\begin{array}{c} \exists F\subset \mathrm{support}(\Gamma),\, |F| = k\\
\forall f \in F\quad f\text{ closed}\end{array}}\leqslant \binom{\ell/2d}{k} (1-f(p,q))^k.
$$
We use this inequality in \eqref{fk.sumpMk} and we obtain 
$$\mfk{TB}{\Lambda}{\begin{array}{c} \exists F\subset \mathrm{support}(\Gamma),\, |F| = k\\
\forall f \in F\quad f\text{ closed}\\
\exists C\in \mathcal{C} \quad d(e,C)\geqslant \ell\end{array}}\\ \leqslant \binom{\ell/2d}{k} (1-f(p,q))^k \fk{TB}{\Lambda}{D} .
$$
Combined with \eqref{fk.decomp}, we have 
$$P_\mu\big(Y_t\in M(k)\big)\leqslant \binom{\ell/2d}{k} (1-f(p,q))^k\cfk{TB}{\Lambda}{D}{T\nconnect B}.
$$
We replace $P_\mu\big(Y_t\in M(k)\big)$ in the previous upper bound for~\eqref{fk.pSTCcapCut} and we have another upper bound as follows:
$$\cfk{TB}{\Lambda}{D}{T\nconnect B}\sum_\Gamma\sum_k\left(\frac{s(1-p)}{|\Lambda|(1+p/q-p)}\right)^{n-k}\binom{\ell/2d}{k} (1-f(p,q))^k.
$$
We calculate the sum and we get
\begin{multline*}\sum_\Gamma\sum_k\left(\frac{s(1-p)}{|\Lambda|(1+p/q-p)}\right)^{n-k}\binom{\ell/2d}{k} (1-f(p,q))^k\\ 
\leqslant \left(\frac{1-p}{p/q}\right)^{n}\sum_\Gamma\sum_k \binom{\ell/2d}{k}\left(\frac{s}{|\Lambda|}\right)^{n-k}\\
\leqslant \sum_\Gamma\left(\frac{1-p}{p/q}\right)^{n}\left(1+\frac{s}{|\Lambda|}\right)^{\ell/2d}\leqslant \left(\frac{\alpha(d)(1-p)}{p/q}\right)^{n}\left(1+\frac{s}{|\Lambda|}\right)^{\ell/2d}.
\end{multline*}
Since $s\leqslant |\Lambda|$, there exists a constant $\tilde{p}<1$ such that 
$$\left(\frac{\alpha(d)(1-p)}{p/q}\right)^{n}\left(1+\frac{s}{|\Lambda|}\right)^{\ell/2d}\leqslant e^{-\ell}.
$$
We obtain the following upper bound for~\eqref{fk.quotien}:
$$\frac{P_\mu\Big(\{e\in \mathcal{P}_{t+s}\}\,\cap\,\{\exists c_t \in \mathcal{C}_t, d(e,c_t) \geqslant \ell \}\Big)}{P_\mu\Big(\exists c_t \in \mathcal{C}_t, d(e,c_t) \geqslant \ell\Big)}\leqslant e^{-\ell}\frac{\cfk{TB}{\Lambda}{D}{T\nconnect B}}{P_\mu\Big(\exists c_t \in \mathcal{C}_t, d(e,c_t) \geqslant \ell\Big)}.
$$
Finally, we notice that 
$$\cfk{TB}{\Lambda}{D}{T\nconnect B} = P_\mu\Big(\exists c_t \in \mathcal{C}_t, d(e,c_t) \geqslant \ell\Big)
$$
and we obtain the desired inequality.
\end{proof}
\section{The interface in the FK-percolation model}
We now prove the main result stated in theorem~\ref{fk.main1}. We follow the main steps of the proof of theorem 1.1 in \cite{new}. We start with the definition of $d^\ell_H$, a semi-distance, similar to the Hausdorff distance, between two subsets $A,B$ of $\Lambda$,
$$ d_H^\ell(A,B) = \inf\,\left\lbrace r\geqslant 0 \,:\begin{array}{c}A\setminus \mathcal{V}(\Lambda^c,\ell)\subset\mathcal{V}(B,r)\\ B\setminus \mathcal{V}(\Lambda^c,\ell)\subset\mathcal{V}(A,r)\end{array}\right\rbrace.
$$
We show a lemma which controls the speed of the pivotal edges. 
\begin{lem}\label{fk.dH}
We have the following result:
\begin{multline*}\exists \tilde{p}<1 \quad \exists \kappa>1 \quad \forall p\geqslant \tilde{p}\quad \forall c \geqslant 1\quad\forall \Lambda\quad |\Lambda|\geqslant 4\quad \forall t \geqslant 0 \\  P_\mu\Big(
\exists s\leqslant |\Lambda|\quad d^{\kappa c\ln|\Lambda|}_H(\mathcal{P}_t,\mathcal{P}_{t+s})\geqslant \kappa c\ln |\Lambda|
 \Big)\leqslant \frac{8d}{|\Lambda|^c}.
\end{multline*}
\end{lem}
\begin{proof}
We fix $s\in \oset 1,\dots, |\Lambda|\cset$. By the definition of the semi-distance $d_H^\ell$, we have, for any $\kappa>1$,
\begin{multline}P_\mu\Big(d_H^{\kappa c\ln|\Lambda|}(\mathcal{P}_t,\mathcal{P}_{t+s})\geqslant \kappa c\ln |\Lambda|\Big)\leqslant \\
P_\mu\Big(\mathcal{P}_{t+s}\setminus \mathcal{V}(\Lambda^c,\kappa c\ln|\Lambda|)\nsubseteq \mathcal{V}(\mathcal{P}_t,\kappa c\ln|\Lambda|)\Big)\\+ P_\mu\Big(\mathcal{P}_t\setminus \mathcal{V}(\Lambda^c,\kappa c\ln|\Lambda|)\nsubseteq \mathcal{V}(\mathcal{P}_{t+s},\kappa c\ln|\Lambda|)\Big).\label{fk.pH}
\end{multline}
Since the two probabilities in the sum depend only on the process $Y$, which is reversible, they are in fact equal to each other. We shall estimate the first probability. By proposition~\ref{fk.vconst}, for any $\ell\geqslant 1$, we have
$$P_\mu\left(\begin{array}{c} e\in\mathcal{P}_{t+s}\\ \exists c_t\in \mathcal{C}_t\quad d(e,c_t)\geqslant \ell\end{array}\right) \leqslant e^{-\ell}.
$$
In order to replace $c_t$ by $\mathcal{P}_t$ in the last probability, we use proposition~\ref{fk.distpivot} for the configuration $Y_t$. We introduce another constant $\kappa'$ and we have 
$$P_\mu \left(\begin{array}{c}\exists C\in\mathcal{C}_t\quad\exists f\in C\\
 d(f,\Lambda^c\cup \mathcal{P}_t\setminus \{f\})\geqslant \kappa' c\ln|\Lambda|\end{array} \right)\leqslant \frac{1}{|\Lambda|^c}.
$$
Therefore, by distinguishing two cases for the configuration $Y_t$, we have
\begin{multline*}
P_\mu\big( e\in \mathcal{P}_{t+s}, d(e,\mathcal{P}_t)\geqslant \kappa c\ln|\Lambda|\big)\leqslant \\ P_\mu\left(\begin{array}{c}e\in \mathcal{P}_{t+s},\quad d(e,\mathcal{P}_t)\geqslant \kappa c\ln|\Lambda|,\\ \forall C\in \mathcal{C}_t\quad\forall f\in C\setminus\mathcal{V}(\Lambda^c,\kappa' c\ln|\Lambda|)\\
 d(f,\mathcal{P}_t)< \kappa' c\ln|\Lambda|
 \end{array}\right)\\
+P_\mu\Big( \exists C\in\mathcal{C}_t,\exists f\in C, d(f,\Lambda^c\cup \mathcal{P}_t\setminus \{f\})\geqslant \kappa' c\ln|\Lambda|\Big).
\end{multline*}
For $\kappa>\kappa'$, we have
\begin{multline*}
P_\mu\left(\begin{array}{c}e\in \mathcal{P}_{t+s}\quad d(e,\Lambda^c\cup\mathcal{P}_t)\geqslant \kappa c\ln|\Lambda|\\ \forall C\in \mathcal{C}_t\quad\forall f\in C\setminus\mathcal{V}(\Lambda^c,\kappa' c\ln|\Lambda|)\\
 d(f,\mathcal{P}_t)< \kappa' c\ln|\Lambda|\end{array}\right)\leqslant  \\
 P_\mu\left(\begin{array}{c} e\in\mathcal{P}_{t+s}\\ \exists c_t\in \mathcal{C}_t\quad d(e,c_t)\geqslant (\kappa-\kappa')c\ln|\Lambda|\end{array}\right) \leqslant \frac{1}{|\Lambda|^{(\kappa-\kappa')c}}.
\end{multline*}
We choose now $\kappa = \kappa'+1$, and we sum over $e$ in $\Lambda$ and $s\in\{1,\dots,|\Lambda|\}$ to get 
$$P_\mu\left(\begin{array}{c}
\exists s\leqslant |\Lambda|, \exists e\in \mathcal{P}_{t+s}\\ d(e,\Lambda^c\cup\mathcal{P}_t)\geqslant \kappa c\ln|\Lambda|
\end{array}  \right)\leqslant \frac{4d}{|\Lambda|^{c-2}}.
$$ 
This is the first probability in~\eqref{fk.pH} and we conclude that
$$P_\mu\left(\begin{array}{c}
\exists s\leqslant |\Lambda|\\d_H^\Lambda(\mathcal{P}_t,\mathcal{P}_{t+s})\geqslant \kappa c\ln |\Lambda|
\end{array}  \right)\leqslant \frac{8d}{|\Lambda|^{c-2}}.
$$ 
Finally, we replace $c$ by $c+2$ and we have 
$$P_\mu\left(\begin{array}{c}
\exists s\leqslant |\Lambda|\\d_H^\Lambda(\mathcal{P}_t,\mathcal{P}_{t+s})\geqslant 3\kappa c\ln |\Lambda|
\end{array}  \right)\leqslant \frac{8d}{|\Lambda|^{c}}.
$$
This is the desired inequality.
\end{proof}
The rest of the proof of theorem~\ref{fk.main1} is exactly the same as the proof of theorem 1.1~in~\cite{new} which relies essentially on lemma~\ref{fk.dH} in \cite{new} independently from the model. We distinguish the case $e\in\mathcal{P}_t$ and the case $e\in\mathcal{I}_t\setminus \mathcal{P}_t$. For the first case, we apply proposition~\ref{fk.distpivot}. For the second case, we notice that this configuration is due to a movement of pivotal edges of distance $\ln^2|\Lambda|$ during a time interval of order $|\Lambda|$. We then apply lemma~\ref{fk.dH} to prove that this event is unlikely.

\section{The interface in the Ising model}
We recall the definition of the two sets $\mathcal{P}_I$ and $\mathcal{I}_I$:
$$\mathcal{P}_I= \left\lbrace \langle x,y\rangle\in E\,:\,\begin{array}{c}
\sigma^D(x) = +1\quad  \sigma^D(y) = -1,\\
x\text{ connected to }T\text{ by a path of spin }+1\text{ in }\sigma^D,\\
y\text{ connected to }B\text{ by a path of spin }-1\text{ in }\sigma^D
\end{array}\right\rbrace.
$$
and
$$\mathcal{I}_I =\left\lbrace\langle x,y\rangle\in E\,: \,\begin{array}{c}\sigma^{+}(x) = \sigma^{+}(y)\\ \sigma^-(x)=\sigma^-(y)\\ \sigma^D(x)\neq\sigma^D(y)\end{array}\right\rbrace. 
$$
We now show our main result on the Ising model.
\begin{proof}[Proof of theorem \ref{fk.main2}]
We fix a constant $c$ and we let $\kappa$ be a constant which will be determined later.
Let us consider the FK configurations $(\omega,\omega')$ associated to $(\sigma^+,\sigma^-,\sigma^D)$. We consider the set of the pivotal edges $\mathcal{P}$ and the set of the interface edges $\mathcal{I}$ in the couple $(\omega,\omega')$. We claim that $\mathcal{P}\subset \mathcal{P}_I$. In fact, for an edge $\langle x,y\rangle\in \mathcal{P}$, one of the endpoint $x$ is connected to $T$ and the other $B$ is connected to $B$ in $\omega'$. Therefore, we have $\sigma^D(x) = 1$ and $\sigma^D(y) = -1$. The configuration $\omega$ dominates $\omega'$, so both of them are also connected to the boundary of $\Lambda$ in $\omega$. We conclude that $$\sigma^+(x) = \sigma^+(y) = 1 \text{ and }\sigma^-(x) = \sigma^-(y) = -1.$$ Let $e$ be an edge in $\mathcal{I}_I$. Since the endpoints of $e$ have different spins in $\sigma^D$, then the edge $e$ must be closed in $\omega'$ and its endpoints belong to two distinct open clusters of $\omega'$. We denote by $x,y$ the two endpoints of $e$ and by $C_x,C_y$ the two open clusters of $\omega'$ such that $x\in C_x$ and $y\in C_y$. The vertices $x$ and $y$ have the same spin in $\sigma^+$ and in $\sigma^-$, therefore one of them, as well as its cluster, has different spins in $\sigma^+$ and $\sigma^D$. Suppose it is $C_x$. Suppose that $e$ belongs to $\mathcal{I}_I\setminus \mathcal{P}_I$. We distinguish three cases as follows:
\begin{itemize}[leftmargin = 0.4cm]
\item If $x$ is connected to the boundary of $\Lambda$ in $\omega'$, then it is connected to the boundary in $\omega$ and the spin of $x$ is determined by the boundary condition in $\sigma^+$ and $\sigma^-$. We have $\sigma^+(x) = +$ and $\sigma^-(x) = -$. Since we have $e\in \mathcal{I}_I$, we have also $\sigma^+(y) = +$ and $\sigma^-(y) = -$. Since the difference of the spin between $\sigma^+$ and $\sigma^-$ can only be induced by the boundary, the vertex $y$ is also connected to the boundary in $\omega$. However, since the edge $e$ is not in $\mathcal{P}_I$, the vertex $y$ can not be connected to the boundary in $\omega'$. Therefore there is an edge $f\in\mathcal{I}$ included in the boundary of the $C_y$.
\item If $x$ is connected to the boundary in $\omega$ but not in $\omega'$. Then, $x$ is connected to an edge $f\in\mathcal{I}$. In other words, there exists an edge $f\in\mathcal{I}$ on the boundary of $C_x$.
\item If $x$ is not connected to the boundary in $\omega$, we have then $\sigma^+(x) = \sigma^-(x)$ and $x$ is not connected to the boundary of $\Lambda$ in $\omega'$. By construction of $\sigma^+$, $C_x$ is contained in an open cluster of $\omega$ which is the union of at least two different open clusters of $\omega'$. Therefore, the boundary of $C_x$ contains at least one edge of $\omega'$ in $\mathcal{I}$.
\end{itemize}
In all cases, at least one endpoint of the edge $e$ is included in an open cluster in $\omega'$ which contains an edge $f\in\mathcal{I}$ on the boundary.  Let us fix the edge $e$ at distance more than $\kappa c \ln^2|\Lambda|$ from $\Lambda^c$. 
We distinguish two cases according to the position of the edge $f\in \mathcal{I}$. Let $\kappa'$ be the constant given by theorem \ref{fk.main1}. The probability in the theorem is less than
\begin{multline}\mu_{p,2}\Big(\exists f \in \mathcal{I}, d(f,\Lambda^c\cup \mathcal{P})\geqslant \kappa' c^2\ln|\Lambda|\Big)\\+\mu_{p,2}\left(\begin{array}{c}\exists f\in \mathcal{I}, d(f,\Lambda^c\cup\mathcal{P})<\kappa' c^2\ln|\Lambda|\\
e \in \mathcal{I}_I,d(e,\Lambda^c\cup\mathcal{P}_I)\geqslant \kappa c^2 \ln^2|\Lambda|\\
\exists x \text{ endpoint of }e\text{ s.t. }\\
f\text{ is on the boundary of } C_x\\\text{ and }C_x\cap\partial\Lambda = \emptyset
\end{array}\right).\label{fk.pInear}
\end{multline}
By theorem \ref{fk.main1}, there exists a $p_1$ such that for $p\geqslant p_1$, we have 
\begin{equation}\mu_{p,2}\Big(\exists f \in \mathcal{I}, d(f,\Lambda^c\cup \mathcal{P})\geqslant \kappa' c^2\ln|\Lambda|\Big)\leqslant \frac{1}{|\Lambda|^c}.\label{fk.pIloin}
\end{equation}
We consider the second case where the edge $f$ is at distance less than $\kappa'c^2\ln^2|\Lambda|$ from $\Lambda^c\cup \mathcal{P}$. Notice that the cluster $C_x$ is of diameter at least $(\kappa-\kappa')c^2\ln^2|\Lambda|$. By taking $ \kappa/2\geqslant\kappa'$, the diameter of $C_x$ is at least $(\kappa c^2\ln^2|\Lambda|)/2$. So the second probability in~\eqref{fk.pInear} is less than 
$$\mu_{p,2}\left(\begin{array}{c}
\exists x\text{ endpoint of }e\quad \exists f\in\Lambda\\ C_x \text{ is of diameter at least} (\kappa c^2\ln^2|\Lambda|)/{2}
\\\text{ and }f\in\partial_e C_x\quad C_x\cap\partial\Lambda = \emptyset
\end{array}\right).
$$
We fix an edge $f\in\Lambda$ and we write 
\begin{multline}\mu_{p,2}\left(\begin{array}{c}
\exists x\text{ endpoint of }e\quad \exists f\in\Lambda\\ C_x \text{ is of diameter at least} (\kappa c^2\ln^2|\Lambda|)/{2}
\\\text{ and }f\in\partial_e C_x\quad C_x\cap\partial\Lambda = \emptyset
\end{array}\right)\\
\leqslant \sum_{f\in\Lambda}\frac{\mfk{TB}{\Lambda,p,2}{\begin{array}{c}\exists C_x \text{ of diameter at least }(\kappa c^2\ln^2|\Lambda|)/{2}\\
f\in\partial_eC_x\quad C_x\cap\partial\Lambda = \emptyset\quad T\nconnect B\end{array}}}{\fk{TB}{\Lambda,p,2}{T\nconnect B}}.\label{fk.quofclose}
\end{multline}
In order to estimate the numerator, we fix the set $\partial_eC_x$ and we denote by $n$ the cardinal of $C_x$. We claim that there exists a set of edges $A$ included in $\partial_e C_x$ which is disjoint from a cut and which is of size at least $n/2$. Let us consider the case where every cut $C$ intersects $\partial_eC_x$ (the claim is true in case where there exists a cut which is disjoint from $\partial_eC_x$). Since $\partial_eC_x$ is the outer edge boundary of an open cluster, it does not contain a pivotal edge.
We apply again the exploration process described in lemma \ref{fk.fbk} starting from $T$ and we reuse the notation $\mathcal{G}$ defined in the proof of the lemma. If the sub-graph $\mathcal{G}$ obtained after exploration contains at least $n/2$ edges of $\partial_e C_x$, then the set $\mathcal{A}_1=E(\mathcal{G})\cap \partial_eC_x$ is disjoint from a cut and is of size at least $n/2$. If not, using the same exploration starting from $B$, we obtain a sub-graph $\mathcal{G}'$ such that 
$$\big(\partial_e C_x\setminus E(\mathcal{G})\big)\subset E(\mathcal{G}').$$
Actually, an edge of the set $\partial_e C_x\setminus E(\mathcal{G})$ contains at least one endpoint connected to $T$, since there are no pivotal edges in $\partial_eC_x$, it is not connected to $B$. By the construction of $\mathcal{G}'$, this edge is contained in $\mathcal{G}'$. Therefore, $\mathcal{A}_2 = \partial_e C_x\cap\mathcal{E}'$ contains at least $n/2$ and is disjoint from a cut 
(see figure \ref{fk.fig:cutCx}).

\begin{figure}[ht]
\centering{
\resizebox{100mm}{!}{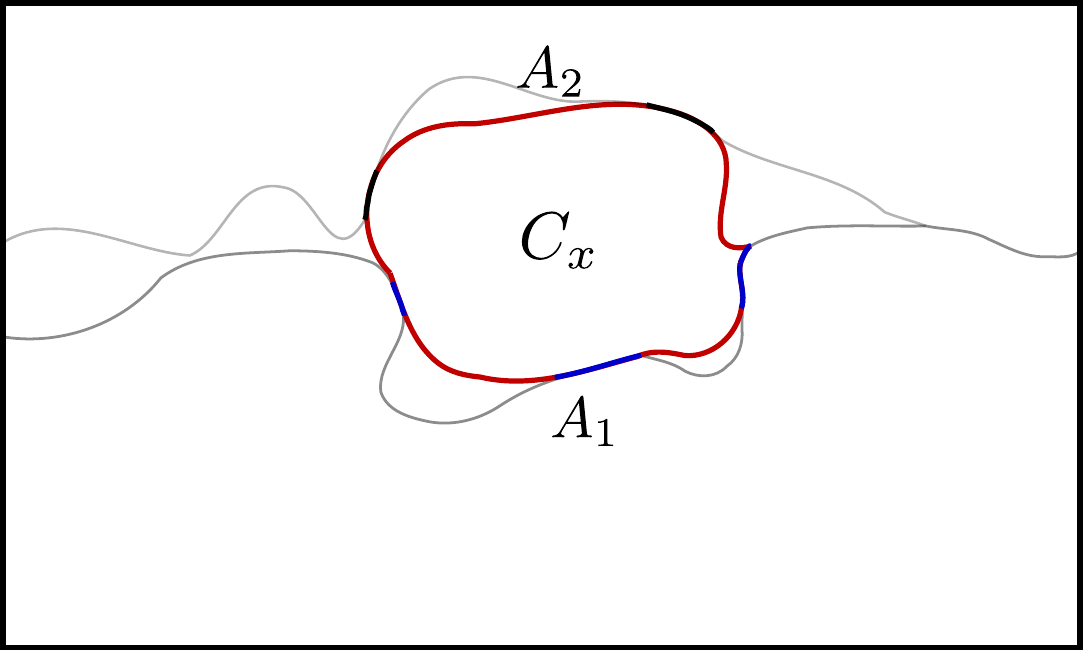}
\caption{The set $\partial_eC_x$ (red, blue and black) and two cuts obtained after the exploration process from $T$ and $B$ (gray). The sets $A_1$ (red and blue) and $A_2$ (red and black) are both disjoint from a cut.}
\label{fk.fig:cutCx}
}
\end{figure} 

For two graphs $G_1$ containing the edges intersecting $B$ and $G_2$ containing the edges intersecting $T$, we set $A_1 = E(G_1)\cap \partial_eC_x$ and $A_2 = E(G_2)\cap \partial_eC_x$. We have
\begin{multline}\mfk{TB}{\Lambda,p,2}{\begin{array}{c}
\forall f \in \partial_e C_x\quad f \text{ is closed}\\
|\partial_eC_x| = n\quad T\nconnect B
\end{array}}\\
\leqslant \sum_{G_1}\mfk{TB}{\Lambda,p,2}{\begin{array}{c}
\forall f \in A_1\quad f\text{ is closed},\\
|A_1| \geqslant n/2,\quad \mathcal{G}= G_1
\end{array}}\\+\sum_{G_2}\mfk{TB}{\Lambda,p,2}{\begin{array}{c}
\forall f \in A_2\quad f\text{ is closed},\\|A_2| \geqslant n/2,\quad\mathcal{G}'= G_2
\end{array}}\\+\mfk{TB}{\Lambda,p,2}{\begin{array}{c}\forall f \in \partial_eC_x \quad f\text{ is closed}\\
|\partial_eC_x| = n\\
\exists C\in\mathcal{C}\quad C\cap \partial_eC_x = \emptyset \end{array}}.\label{fk.sumcluster}
\end{multline}
By the spatial Markov property and the comparison inequality used in the proof of lemma \ref{fk.fbk}. We have
$$\sum_{G_1}\Phi^{TB}_{\Lambda,p,2}\left(\begin{array}{c}
\forall f \in A_1\quad f\text{ is closed},\\
|A_1| \geqslant n/2,\quad \mathcal{G}= G_1
\end{array}\right)
\\ \leqslant (1-f(p,2))^{n/2}\sum_{G_1}\fk{TB}{\Lambda,p,2}{\mathcal{G}=G_1}.
$$
The same goes for the second sum of \eqref{fk.sumcluster}.
The last sum is less than $\fk{TB}{\Lambda,p,2}{T\nconnect B}$.
By lemma \ref{fk.fbk}, last term in \eqref{fk.sumcluster} is less than 
$$2(1-f(p,2))^{n}\fk{TB}{\Lambda,p,2}{T\nconnect B}.
$$
We obtain an upper bound for \eqref{fk.sumcluster} as follows:
$$\mfk{TB}{\Lambda,p,2}{\begin{array}{c}
\forall f \in \partial_e C_x\quad f \text{ is closed}\\
|\partial_eC_x| = n\quad T\nconnect B
\end{array}}\leqslant 4(1-f(p,2))^{n}\fk{TB}{\Lambda,p,2}{T\nconnect B}.
$$
The set $\partial_eC_x$ is $*$-connected (see \cite{Pisztora1996Surface}). For a fixed edge $f$, the number of the $*$-connected sets $\partial_eC_x$ of size $n$ containing the edge $f$ is bounded from above by
$C_n\alpha(d)^n,$
where 
$$C_n = \frac{1}{n+1}\binom{2n}{n}$$ is the nth number Catalan number. Using Stirling formula, an upper bound of the number of $\partial_eC_x$ is $4^n\alpha(d)^n$. We would like to mention the arguments of \cite{MR876084}[p.82] and \cite{grimmett1999percolation}[theorem 4.20] for an upper bound of lattice animals in $\mathbb{Z}^d$.
We obtain therefore
\begin{multline*}\mfk{TB}{\Lambda,p,2}{\begin{array}{c} \exists C_x \text{ of diameter at least }(\kappa c^2\ln^2|\Lambda|)/2 \\
f\in\partial_eC_x \quad \partial_eC_x\cap \partial\Lambda = \emptyset\quad T\nconnect B
\end{array}}\\ 
\leqslant \sum_{n\geqslant (\kappa c^2\ln^2|\Lambda|)/2 }4^{n+1}\alpha(d)^n(1-f(p,2))^{n/2}\fk{TB}{\Lambda,p,2}{T\nconnect B}\\
\leqslant \big(8\alpha^2(d)-8\alpha^2(d)f(p,2)\big)^{\kappa c^2\ln^2|\Lambda|)/4}\fk{TB}{\Lambda,p,2}{T\nconnect B}.
\end{multline*}
We sum over the choices of the edge $f$ and we obtain an upper bound for \eqref{fk.quofclose} as follows:
\begin{multline*}\mu_{p,2}\left(\begin{array}{c}
\exists x\text{ endpoint of }e\quad \exists f\in\Lambda\\ C_x \text{ is of diameter at least} (\kappa c^2\ln^2|\Lambda|)/{2}
\\\text{ and }f\in\partial_e C_x\quad C_x\cap\partial\Lambda = \emptyset
\end{array}\right)\\\leqslant |\Lambda|\big(8\alpha^2(d)-8\alpha^2(d)f(p,q)\big)^{(\kappa c^2\ln^2|\Lambda|)/4}
\end{multline*}
There exist $p_2<1$ and $\kappa\geqslant 1$ such that for $p\geqslant p_2$,
$$|\Lambda|\big(8\alpha^2(d)-8\alpha^2(d)f(p,q)\big)^{(\kappa c^2\ln^2|\Lambda|)/4}\leqslant \frac{1}{|\Lambda|^c}.
$$
Combined with  \eqref{fk.pInear}, \eqref{fk.pIloin}, for 
$p\geqslant \max(p_1,p_2)$, we have 
$$\pi_{\beta}\Big(e \in \mathcal{I}_I,d(e,\Lambda^c\cup\mathcal{P}_I)\geqslant \kappa c^2 \ln^2|\Lambda|\Big)\leqslant \frac{2}{|\Lambda|^c}.
$$
We then sum over the edges $e$ in $\Lambda$ and we get
$$\pi_{\beta}\Big(\exists e \in \mathcal{I}_I,d(e,\Lambda^c\cup \mathcal{P}_I)\geqslant \kappa c^2 \ln^2|\Lambda|\Big)\leqslant \frac{2}{|\Lambda|^{c-1}}.
$$
For $|\Lambda|\geqslant 4$, we can replace $2/|\Lambda|^{c-1}$ by $1/|\Lambda|^{c-2}$. By taking $\kappa$ big enough, we can replace $c-2$ by $c$ and we obtain the result announced in the theorem.
\end{proof}
\section{Proof of the theorem \ref{fk.main3}}
\begin{proof}
By symmetry, it is sufficient to show the first inequality. The proof follows the same arguments for the proof of theorem \ref{fk.main2} and we use the same notations as in the previous proof. We fix a vertex $x$ and we notice that $x$ has different spins only when it is connected to the boundary or it is contained in an open cluster $C_1$ in $\omega$ which is the union of at least two open clusters of $\omega'$. Actually, if $x$ is connected to $T$ in $\omega$, since $\sigma^D(x) = -1$, it is not connected to $T$ in $\omega'$. If $x$ is connected to $B$ in $\omega$, since the cut separates $x$ from $B$ in $\omega'$, it can not be connected to $B$ in $\omega'$. If $x$ is not connected to the boundary of the box, since $\sigma^+(x)\neq \sigma^D(x)$, the open cluster $C_x$ in $\omega$ is the union of at least two open clusters of $\omega'$. In all the three cases, the open cluster of $C_x$ in $\omega'$ contains an edge $f\in\mathcal{I}$ and $C_x$ does not touch the boundary of $\Lambda$. We obtain 
\begin{multline*}\pi_{\beta}\left(\begin{array}{c}
\sigma^+(x) = +1,\quad \sigma^D(x) = -1\\
\exists C \text{ a cut separating }x\text{ from }B\\
d(x,C)\geqslant \kappa c^2\ln^2|\Lambda|
\end{array}\right)\\
\leqslant \mu_{p,2}\left(\begin{array}{c}\exists f\in \mathcal{I} \quad f\in \partial_eC_x\\
\exists C \in\mathcal{C}\quad  d(x,C)\geqslant \kappa c^2\ln^2|\Lambda|\\C_x\cap \partial\Lambda = \emptyset
\end{array}\right).
\end{multline*}
The rest of the proof follows exactly the same arguments as in the previous proof. We distinguish two cases, if there is an edge $f\in\mathcal{I}$ far from the cut $C$, we apply the theorem \ref{fk.main1}. If all the edges of $\mathcal{I}$ are close to the cut, the cluster $C_x$ has a diameter at least $\kappa c^2\ln^2|\Lambda|/2$, the same reasoning starting from \eqref{fk.quofclose} can be applied to obtain an upper bound in this case. Combining the two cases and we have 
$$\pi_{\beta}\left(\begin{array}{c}
\sigma^+(x) = +1,\quad \sigma^D(x) = -1\\
\exists C \text{ a cut separating }x\text{ from }B\\
d(x,C)\geqslant \kappa c^2\ln^2|\Lambda|
\end{array}\right)\leqslant \frac{1}{|\Lambda|^c}.
$$
We sum over the choices of $x$ and we have 
$$\pi_{\beta}\left(\begin{array}{c}\exists x \in\Lambda\quad
\sigma^+(x) = +1,\quad \sigma^D(x) = -1\\
\exists C \text{ a cut separating }x\text{ from }B\\
d(x,C)\geqslant \kappa c^2\ln^2|\Lambda|
\end{array}\right)\leqslant \frac{1}{|\Lambda|^{c-1}}.
$$
We can change $c$ by $c+1$ and we obtain the desired result.
\end{proof}

\bibliographystyle{halpha}
\bibliography{ln2fk}
\end{document}